\documentclass[12pt, german]{article}

\makeatletter
\newfont{\footsc}{cmcsc10 at 8truept}
\newfont{\footbf}{cmbx10 at 8truept}
\newfont{\footrm}{cmr10 at 10truept}
\makeatother
\usepackage{amsmath}
\usepackage{amssymb}
\usepackage{exscale}
\usepackage{amsthm}
\usepackage{epsfig}
\usepackage{setspace}
\usepackage{pgfplots}
\usepackage{tikz}
\usepackage{tikz-cd}
\usetikzlibrary{quotes,angles,arrows,positioning}
\usetikzlibrary{arrows,shapes,positioning}
\usetikzlibrary{calc,decorations.markings}
\usepackage{verbatim}
\usepackage{caption}
\usepackage{ulem}
\usepackage{dsfont}

\theoremstyle{plain}
\newtheorem{theorem}{Theorem}[section]
\newtheorem{proposition}[theorem]{Proposition}
\newtheorem{lemma}[theorem]{Lemma}
\newtheorem{corollary}[theorem]{Corollary}

\theoremstyle{definition}
\newtheorem{definition}[theorem]{Definition}

\newtheorem{remark}[theorem]{Remark}

\newcommand{\N}{{\mathbb{N}}}

\newcommand{\R}{{\mathbb{R}}}

\def\bs0{\bf 0}

\newtheorem{innercustomgeneric}{\customgenericname}
\providecommand{\customgenericname}{}
\newcommand{\newcustomtheorem}[2]{%
	\newenvironment{#1}[1]
	{%
		\renewcommand\customgenericname{#2}%
		\renewcommand\theinnercustomgeneric{##1}%
		\innercustomgeneric
	}
	{\endinnercustomgeneric}
}

\newcustomtheorem{customthm}{Theorem}
\newcustomtheorem{customcor}{Corollary}

\title{A Generalized Faulhaber Inequality, Improved Bracketing Covers, and Applications to Discrepancy}

\author{Michael Gnewuch\thanks{Institut f\"ur Mathematik, Universit\"at Osnabr\"uck, 
Germany ({\tt michael.gnewuch@uni-osnabrueck.de}).}
\and{Hendrik Pasing\thanks{Institut Naturwissenschaften, Hochschule Ruhr West, Germany  ({\tt hendrik.pasing@hs-ruhrwest.de}).}}
\and{Christian Wei{\ss}\thanks{Institut Naturwissenschaften, Hochschule Ruhr West, Germany  ({\tt christian.weiss@hs-ruhrwest.de}).}}
}

\begin{document}

\maketitle
\vskip 1pc

\begin{abstract}
We prove a generalized Faulhaber inequality to bound the sums of the $j$-th powers of the first $n$ (possibly shifted) natural numbers. With the help of this inequality we are able to improve the known bounds for bracketing numbers of $d$-dimensional axis-parallel boxes anchored in $0$
(or, put differently, of lower left orthants intersected with the $d$-dimensional unit cube $[0,1]^d$). We use these improved bracketing numbers to establish new bounds for the star-discrepancy of negatively dependent random point sets and its expectation. 
We apply our findings also to the weighted star-discrepancy. 
\end{abstract}


\section{Introduction} 
Already in the 17th century, Faulhaber found a formula to express the $p$-th power of the first $n$ positive integers. His work did not get much recognition until it was rediscovered in the second half of the 20th century, see most notably \cite{Knu93}. Nowadays, Faulhaber's formula is a standard tool in $p$-adic analysis and a research topic on its own, see e.g. \cite{Lev19} and \cite{BMV20} for two recent examples. In this paper, we demonstrate an application of Faulhaber sums in a new context: we derive a generalized Faulhaber inequality to improve known bounds for bracketing numbers of $d$-dimensional axis-parallel boxes anchored in $0$. Recall that Faulhaber's formula is for $j \in \mathbb{N}$ given by
$$\sum_{i=1}^n i^j = \frac{n^{j+1}}{j+1} + \frac{n^j}{2} + \sum^j_{k=2} (-1)^k \frac{B_k}{k!} (j)_{k-1}n^{j-k+1},$$
where $B_k$ is the $k$-th Bernoulli number and $(j)_{k-1}:=\frac{j!}{(j-k+1)!}$ the Pochhammer symbol. Also note that all odd Bernoulli numbers with $k \geq 3$ are equal to $0$. Furthermore, an easy calculation yields that the coefficient of the power $n^{j-1}$ on the right hand side is $\tfrac{j}{12}$.\\[12pt]
However, we are not faced with simple sums of positive integers in the context of bracketing numbers. Instead a shifted version, namely sums of the form $\sum_{i=1}^n (i+r)^j$ with some shift $0 \le r \le 1$, appears. As a first main result we therefore derive the following 
\textbf{generalized Faulhaber inequality}: 
	\begin{equation*}
	\sum_{i=1}^n (i+r)^j \leq \frac{(n+r)^{j+1}}{j+1} + \frac{(n+r)^j}{2} + \frac{j(n+r)^{j-1}}{12},
	\end{equation*}
see Theorem~\ref{prop:Faulhaber}. Note that the inequality only involves positive coefficients on both sides. This is an essential technical feature which we need in our analysis of bracketing numbers. However, its potential use is not limited to the topic of this paper but might be applied to improve other (discrepancy) bounds.\\[12pt]

Let us now describe the other main results of this paper in more detail. To this purpose we have to state the necessary definitions and concepts.

\paragraph{Bracketing numbers.} The concept of bracketing numbers is well known in empirical process theory, see, e.g., \cite{vdVW96}.
	We use here the definitions and notations as presented in \cite{DGS05} and \cite{Gne08}: Let $A \subset [0,1]^d$. For any $\delta \in (0,1]$ a finite set of points $\Gamma \subset [0,1]^d$ is called a \textbf{$\delta$-cover of $A$}, if for every $y \in A$, there exist $x,z \in \Gamma \cup \left\{ 0 \right\}$ such that $x \leq y \leq z$ and $\lambda_d([0,z]) - \lambda_d([0,x]) \leq \delta$, where  $\lambda_d$ denotes the $d$-dimensional Lebesgue-measure. For $x,y \in \mathbb{R}^d$, the expression $x \leq y$ is meant as a component-wise inequality. If $x \leq y$ the interval $[x,y]$ is defined by $[x,y] := [x_1,y_1] \times [x_2,y_2] \times \ldots \times [x_d,y_d]$. We use corresponding notation for (half-)open $d$-dimensional intervals. The number $N(A,\delta)$ denotes the smallest cardinality of a $\delta$-cover in $A$. The \textbf{bracketing number $N_{[\ ]}(A,\delta)$ of $A$} is the smallest number of closed axis-parallel boxes (or ``brackets'') of the form $[x,y]$ with $x,y \in [0,1]^d$, satisfying $\lambda_d([0,y]) - \lambda_d([0,x]) \le \delta$,
	whose union contains $A$.
	It is easily checked that the  numbers $N(A,\delta)$ and $N_{[\ ]}(A,\delta)$ satisfy the relations
	$$N(A,\delta) \le 2N_{[\ ]}(A,\delta) \le  (N(A,\delta) + 1)N(A,\delta).$$
	We will rely on the first inequality. Furthermore, we set $N_{[\ ]}(d,\delta):=N_{[\ ]}([0,1]^d,\delta)$. In dimension $d=1$, the identity $N_{[\ ]}(1,\delta) = \lceil \delta^{-1} \rceil = N([0,1]^d, \delta)$ holds as can be seen by the help of the $\delta$-cover $\Gamma:=\left\{ 1/ \lceil \delta^{-1} \rceil, 2 /\lceil \delta^{-1} \rceil, \ldots \right\}$. For higher dimension, the following theorem was proven in \cite{Gne08}.

\begin{theorem}[\cite{Gne08}, Theorem 1.15] \label{thm:brackgen1} For any $d \geq 1$ and $1\ge \delta > 0$ we have
	$$N_{[\ ]}(d,\delta) \leq 2^{d-1} \frac{d^d}{d!} (\delta^{-1}+1)^d.$$
	
\end{theorem}
The result was further improved in \cite{PW20} to:
\begin{theorem}[\cite{PW20}, Proposition 2.3] \label{thm:brackgen2} For any $d \geq 1$ and $1\ge \delta > 0$ we have
	$$N_{[\ ]}(d,\delta) \leq 2^{d-2} \frac{d^d}{d!} (\delta^{-1}+1)^d + \frac{1}{2} \left( \delta^{-1} + 1 \right).$$
\end{theorem}

In this paper, we establish a smaller upper bound for the bracketing number in dimension $d=2$ by an explicit geometric construction (Theorem~\ref{thm:brack2}). By comparing our proof to the simple argument in dimension $d=1$, it can be seen that direct geometric arguments to approximate bracketing numbers get more and more complicated as the dimension increases. Still we are able to significantly improve the general upper bound from Theorem~\ref{thm:brackgen1} and Theorem~\ref{thm:brackgen2}, respectively, for all $d\in \N$ to
	\begin{equation*} 
	N_{[\ ]}(d, \delta) \le \max\left({1.1}^{d-101},1\right) \frac{d^d}{d!} (\delta^{-1} + 1)^d,
	\end{equation*}
see Theorem~\ref{thm:general_bracketing_number}.
Note that the involved power of $2^k$ with $k=d-1$ and $k=d-2$, respectively, is reduced to $\max\left( 1, {1.1}^{d-101} \right)$. The proofs of Theorem~\ref{thm:brackgen1} and Theorem~\ref{thm:brackgen2} utilized simple integral criteria for an estimation of the bracketing number. Our improvements significantly rely on the use of the generalized Faulhaber inequality, Theorem~\ref{prop:Faulhaber}.

\paragraph{Discrepancy.} One advantage of $\delta$-covers is that they can be used to approximate the star-discrepancy of a sequence or set. 
Let $P$ be an $N$-point set in $[0,1)^d$. 
Then the \textbf{star-discrepancy} of $P$ is defined by
$$D^*_N(P) := \sup_{B \subset [0,1)^d} \left| \frac{1}{N}|P\cap B| 
- \lambda_d(B) \right|,$$
where the supremum is taken over all intervals $B = [0,a) \subset [0,1)^d$ and where 
$|A|$ denotes the number of elements in a finite set $A$.
For more details we refer the reader to \cite{DP10}, 
\cite{Nie92}. As a link to bracketing numbers we have the following lemma which is easy to prove. 

\begin{lemma}[\cite{DGS05}, Lemma~3.1]\label{Lemma:3.1} 
Let $P \subset [0,1)^d$ be an $N$-point set, $\delta > 0$, and $\Gamma$ be a $\delta$-cover of $[0,1]^d$. Then
	$$D_N^*(P) \leq \max_{x \in \Gamma} D_N(P,[0,x)) + \delta,$$
	where 
	$$D_N(P,[0,x)) := \left| \frac{1}{N} |P \cap [0,x)| - \lambda_d([0,x)) \right|.$$
\end{lemma} 

Lemma~\ref{Lemma:3.1} can be useful for several applications: It can be applied to prove (theoretical) upper discrepancy bounds (see., e.g., \cite{Ais11, Gne12,NP18}  and Section~\ref{SUBSEC:Star_Discrepancy}), to construct algorithms that approximate the discrepancy of given input sets $P$ up to a user-specified error tolerance (see., e.g., \cite{Thi01, DGW14}) and algorithms that generate point sets that exhibit a small star-discrepancy (see, e.g., \cite{DGW09, DGW10}). Similar results based on $\delta$-covers of $[0,1]^d$ are known for other quality criteria for $N$-point sets, as, e.g., the weighted star-discrepancy (cf. \cite{Ais14, HPS08, WGH19} and Section~\ref{SUBSEC:Weighted_Star_Discrepancy}), the discrepancy with respect to all axis-parallel boxes, sometimes also called \textbf{extreme discrepancy} (see, e.g., \cite{Gne08}), or the \textbf{dispersion} (see, e.g., \cite{Rud18, HKKR20}), a quality measure that has important applications in optimization theory (cf., e.g., \cite{BGKTV17, Nie92}).
\\[12pt]
In all these applications it is essential to have good bounds for the bracketing number. 
Moreover, for constructing algorithms explicit  bracketing or $\delta$-covers of small size are essential; the smaller the $\delta$-covers, the faster the running time of the algorithms (or, alternatively, the larger the size of the instances that can be treated). That is why improved cardinality bounds and constructions of $\delta$-covers are highly relevant.
 \\[12pt]
The star-discrepancy is closely related to the construction of optimal integration rules via the Koksma-Hlawka inequality, namely: For every $N$-point set $P \subset [0,1)^d$ we have
$$\left| \int_{[0,1]^d} f(x) \mathrm{d}\lambda_d(x) - \frac{1}{N} \sum_{p \in P} f(p) \right| \leq D_N^*(P) V(f),$$
where $V(f)$ is the variation of $f$ in the sense of Hardy and Krause. If all partial mixed derivatives of $f$ are continuous on $[0,1]^d$ then  
$$V(f) = \sum_{u} \int_{[0,1]^u} \left| \tfrac{\partial^{|u|}f}{\partial x_u}(x_u,1) \right| \mathrm{d}x_u,$$ 
where the sum is taken over all subsets $u \subset \left\{ 1,2,\ldots,d \right\}$ and $(x_u,1)$ is the $d$-dimensional vector whose $i$-th component is $x_i$ if $i \in u$ and $1$ otherwise, see \cite{KN74}, Chapter~2. It follows from the Koksma-Hlawka inequality that integration rules with small star-discrepancy yield small integration errors. The task of high-dimensional integration occurs in different applications, e.g. in computational finance (see e.g. \cite{CMO97}, \cite{GW09}, \cite{Gla03}, \cite{Pas94},
\cite{WN19}). Therefore, it is of interest to know sharp bounds for the smallest achievable star-discrepancy
$$D^*(N,d) := \inf \left\{ D_N^*(P) | P \subset [0,1)^d, |P| = N \right\}$$
or, equivalently, for the \textbf{inverse of the star-discrepancy}
$$N^*(\varepsilon,d):= \inf \left\{ N \in \mathbb{N} | D^*(N,d) \leq \epsilon \right\}$$
which is the minimum number of sample points that guarantees a discrepancy bound of at most $\varepsilon$. If an $N$-point set $P \subset [0,1)^d$ exhibits a star-discrepancy of order
\begin{equation}\label{eq:asymp_bound}
D_N(P) = O(N^{-1}(\log N)^{d-1}),
\end{equation}
then it is called a \textbf{low-discrepancy point set}. It is conjectured that this is the best possible rate of convergence. In fact, this is trivially true in dimension one and it is known to be true for dimensions two by the work of Schmidt, \cite{Sch72}. 
For arbitrary dimension $d$ lower bounds show that \eqref{eq:asymp_bound}
is sharp up to logarithmic factors, see \cite{BL08}, \cite{BLV08}.
Due to the exponential dependence of the asymptotic bound \eqref{eq:asymp_bound}
on $d$ it does not  provide helpful information about the star-discrepancy of sample sets of moderate size, cf., e.g., \cite{Gne12}. Therefore, researchers started to prove pre-asymptotic bounds for the star-discrepancy.
The best known upper and lower bounds for the smallest achievable star-discrepancy with explicitly given dependence on the number of points as well as on the dimension are of the form
$$D_N^*(P) \leq C \sqrt{\frac{d}{N}}$$
for some constant $C$. This implies
\begin{equation}\label{est:inv_star_disc}
N^*(\varepsilon,d) \leq \lceil C^2 d \varepsilon^{-2} \rceil
\end{equation}
for all $d, N \in \mathbb{N}$ and $\varepsilon \in (0,1]$. The theoretical existence of such a $C$ was shown in \cite{HNWW01} but no concrete value for $C$ was calculated therein. In \cite{Ais11}, a new proof of the result was given including the first explicit upper bound $C=10$ (or more precisely $C=9.65$). In \cite{PW20}, the value was improved to $C=9$,  in a previous version of the recent preprint \cite{Doe16} to $C=2.7868$, and in \cite{GH16} even to $C=2.5287$. In the course of this paper, we will show that indeed $C \leq 2.4968$ holds (Corollary~\ref{cor:improvedC}). Even small improvements of the constant $C$ are not only of pure theoretical interest but lead to a smaller necessary number of points $N$ to guarantee a given level of precision for integration (observe that  due to \eqref{est:inv_star_disc} $N$ scales quadratically in $C$). Moreover the constant can be used to check whether a given (random or deterministic) point set or sequence is good in the sense that its empirically observed star-discrepancy is close to or even smaller than $C\sqrt{\frac{d}{N}}$. For instance, the authors of \cite{GGP20} conclude that the sequence generated by a secure bit generator is good up to dimension at least $d=15$, because its empirically observed discrepancy is smaller than $1 \cdot \sqrt{\frac{d}{N}}$ also for relatively big $N$. 

It was proved in \cite{Hin04} that there exist $c,\varepsilon_0>0$ such that 
$$N^*(\varepsilon,d) \geq cd\varepsilon^{-1}$$
for all $0 < \varepsilon \leq \varepsilon_0, d \in \mathbb{N}$. Note that there is a gap between the lower and the upper bound on $N^*(\varepsilon,d)$ with regard to the dependence on $\varepsilon^{-1}$. In \cite{Doe13}, it was shown that this gap does not exist for \textit{typical} Monte Carlo points. More precisely, there exists a constant $K>0$ such that the expected star-discrepancy of a Monte Carlo $N$-point sample fulfills
\begin{equation}\label{doe1}
\mathds{E}[D_N^*(X)] \geq K \sqrt{\frac{d}{N}}
\end{equation}
and additionally there is the probabilistic error bound
\begin{equation}\label{doe2}
\mathds{P} \left(D_N^*(X) <  {K}\sqrt{\frac{d}{N}} \right) \leq \exp(-\Theta(d)).
\end{equation}
An analogous result for Latin hypercube sampling instead of Monte Carlo sampling holds for dimension $d \geq 2$, see \cite{DDG18}. In the course of this paper we complement estimate \eqref{doe1} by showing that Monte Carlo $N$-point samples satisfy 
\begin{equation} \label{eq:intro:1}
\mathds{E}[D_N^*(X)] \le 2.55648 \sqrt{\frac{d}{N}},
\end{equation}
see Corollary~\ref{cor:expectation}. This is an immediate consequence of Proposition~\ref{Prop:Expec_Disc}, which shows how to bound the expected discrepancy of arbitrary random points with the help of probabilistic discrepancy bounds. Estimate \eqref{eq:intro:1} may serve as an indicator for a given random point set whether its discrepancy is better or worse than average. We close this paper by studying the so-called \textbf{weighted star-discrepancy} and derive probabilistic upper bounds for it, see Proposition~\ref{Prop:Pre_Aistleitner} as well as Corollaries~\ref{Ais_weighted_Discrepancy} and \ref{Cor:WGH19+}.
The general upper bound in Proposition~\ref{Prop:Pre_Aistleitner} gives immediately the upper 
bound in Corollary~\ref{Ais_weighted_Discrepancy}, which results in an improvement of a result of Aistleitner, see \cite[Theorem~1]{Ais14}.
Furthermore, Corollary~\ref{Cor:WGH19+} improves a bound from \cite{WGH19}.

\section{Improved Bracketing Number}

In this section, we derive new upper bounds for bracketing numbers. In dimension $d=2$, we do this by an explicit geometric construction. Roughly speaking, we get an additional factor $\log(2)$ implying an improvement of approximately $30 \%$ of the highest order constant in comparison to \cite{PW20}, Proposition 2.3. For arbitrary higher dimension we are able to reduce the involved power of $2^k$ from $2^{d-1}$ in Theorem~\ref{thm:brackgen1} and $2^{d-2}$ in Theorem~\ref{thm:brackgen2}, respectively, to $\max\left( 1, 1.1^{d-101} \right)$ in Theorem~\ref{thm:general_bracketing_number}.

\subsection{Constructive Bound in Dimension $d = 2$}\label{SubSec:2.1}

While it is straightforward to calculate the bracketing number for dimension $d=1$, namely $N_{[\ ]}(1,\delta) = \lceil \delta^{-1} \rceil$, the task is already much harder in dimension $2$. The general construction of the bracketing cover for our improved bound of $N_{[ \ ]}(2,\delta)$ follows the lines of \cite{EJoC08}. Here we provide a more careful analysis as in \cite{EJoC08}
to derive a bound on $N_{[ \ ]}(2,\delta)$ with explicitly given constants. The construction goes as follows: 

In a first step, we cover the diagonal of $[0,1]^2$ with $\delta$-brackets by choosing the points $\underline{a}_q:=(a_q,a_q)$, $0 \le q \le n-1$, as points of our $\delta$-cover, where $a_q:= \sqrt{1-q\delta}$ and $n:= \lceil \delta^{-1} \rceil$. Furthermore, we put $\underline{a}_n:= (0,0)$. Then 
$$
\lambda_2 ([0, \underline{a}_{q-1}]) - \lambda_2 ([0, \underline{a}_{q}]) = \delta
\hspace{3ex}\text{for all $1\le q\le n$.}
$$
We observe that 
\begin{align} \label{ineq:brack_decomposition}
	N_{[ \ ]}(2,\delta) \leq \sum_{q=1}^{n} N_{[ \ ]}([0,\underline{a}_{q-1}] \setminus [0,\underline{a}_{q}],\delta).
\end{align}

In a second step, we cover each of the sets $[0,\underline{a}_{q-1}] \setminus [0,\underline{a}_{q}]$, $1\le q \le n-1$, with $\delta$-brackets and add the upper right corner points of the brackets to our $\delta$-cover $\Gamma$. The construction of our $\delta$-cover leads to the upper bound on 
$N_{[ \ ]}([0,\underline{a}_{q-1}] \setminus [0,\underline{a}_{q}],\delta)$ presented in 
Lemma~\ref{lem:elementary_brackets} below; the construction itself is described in detail in the proof of Lemma~\ref{lem:elementary_brackets}.

\begin{center}
	\begin{tikzpicture}[scale=0.8]
		\draw[thick] (0,0)--(10,0);
		\draw[thick] (10,0)--(10,10);
		\draw[thick] (10,10)--(0,10);
		\draw[thick] (0,10)--(0,0);
		
		\draw[thick] (0,8.5)--(10,8.5);		
		
		\draw[fill=black] (10,10) circle (.5ex) node[above] {$k_1$};
		\draw[fill=black] (8.5,10) circle (.5ex) node[above] {$k_2$};
		\draw[fill=black] (6.9,10) circle (.5ex) node[above] {$k_3$};
		\draw[fill=black] (5.1,10) circle (.5ex) node[above] {$k_4$};
		\draw[fill=black] (2.8,10) circle (.5ex) node[above] {$k_5$};
		\draw[fill=black] (0.4,10) circle (.5ex) node[above] {$k_6=k_f$};
		
		\draw[fill=black] (8.5,8.5) circle (.5ex) node[below] {$l_1$};
		\draw[fill=black] (6.9,8.5) circle (.5ex) node[below] {$l_2$};
		\draw[fill=black] (5.1,8.5) circle (.5ex) node[below] {$l_3$};
		\draw[fill=black] (2.8,8.5) circle (.5ex) node[below] {$l_4$};
		\draw[fill=black] (0.4,8.5) circle (.5ex) node[below] {$l_5$};		
		
		\draw (10,0.2)--(10,-0.2)  node[below] {$b_0=1$};
		\draw (8.5,0.2)--(8.5,-0.2) node[below] {$b_1$};
		\draw (6.9,0.2)--(6.9,-0.2) node[below] {$b_2$};
		\draw (5.1,0.2)--(5.1,-0.2) node[below] {$b_3$};
		\draw (2.8,0.2)--(2.8,-0.2) node[below] {$b_4$};
		\draw (0.4,0.2)--(0.4,-0.2) node[below] {$b_5$};
		\draw[fill=black] (0,10) circle (.0ex) node[left] {$1 = a_0$};
		\draw[fill=black] (0,8.5) circle (.0ex) node[left] {$b_1 = a_1$};
		\draw[fill=black] (0,0) circle (.5ex) node[left] {$0= l_6 = l_f$};
			
		\draw[dashed] (10,10)--(8.5,8.5);
		\draw[dashed] (8.5,10)--(6.9,8.5);
		\draw[dashed] (6.9,10)--(5.1,8.5);
		\draw[dashed] (5.1,10)--(2.8,8.5);
		\draw[dashed] (2.8,10)--(0.4,8.5);
		
		\draw[dashed] (8.5,10)--(8.5,8.5);
		\draw[dashed] (6.9,10)--(6.9,8.5);
		\draw[dashed] (5.1,10)--(5.1,8.5);
		\draw[dashed] (2.8,10)--(2.8,8.5);
		\draw[dashed] (0.4,10)--(0.4,8.5);
	\end{tikzpicture} 
	
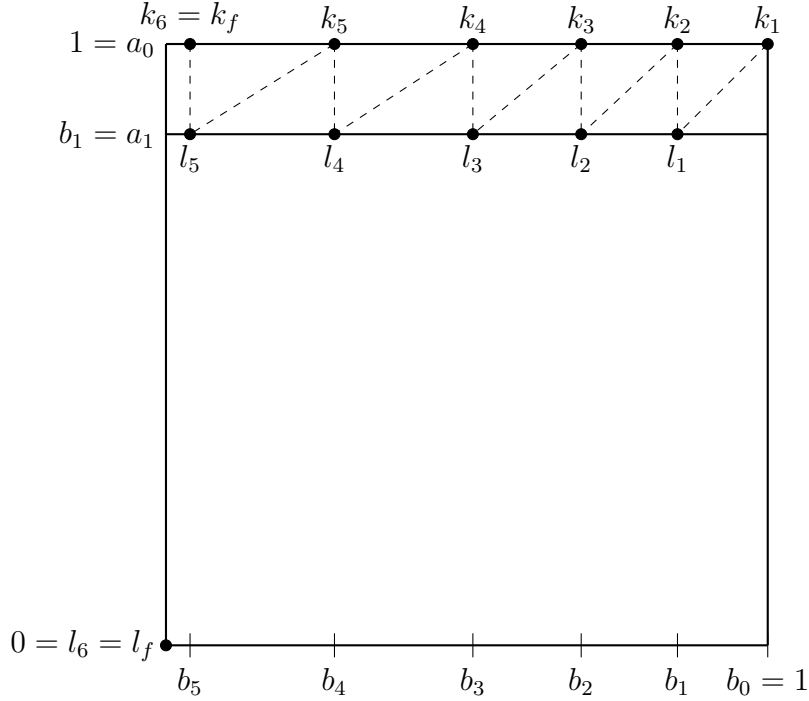
\captionof{figure}{ Unit square with marked points} \label{fig1}
\end{center}
\begin{lemma} \label{lem:elementary_brackets} 
Let $0 < \delta < 1$ be arbitrary and choose $q \leq n$. For $\delta_q:=\tfrac{\delta}{1-(q-1)\delta}$ we have
	$$N_{[ \ ]}([0,\underline{a}_{q-1}] \setminus [0,\underline{a}_{q}],\delta) \leq 2 \left\lceil \frac{ -2 \cdot \log(2)}{\log\left(1 - \delta_q \right)} \right\rceil + 1 =: 2f(\delta_q) - 1. $$
For $q=n$ we even have 	$N_{[ \ ]}([0,\underline{a}_{n-1}] \setminus \{0\}, \delta) =1$.
\end{lemma}
\begin{proof} At first, we consider the case $q = 1$. Then we define $b_0:=1$ and let $b_i$ be $1$ minus the sum over the widths of the $i$ first $\delta$-brackets covering $[0,1]^2 \setminus ([0,1] \times [0,a_1])$. Furthermore, let $k_i:=(b_{i-1},1)$ and $l_i:=(b_i,b_1)$ be the points as indicated in Figure~\ref{fig1}. We choose the $b_i$ such that the condition
\begin{align} \label{eq1}
	\lambda_2([0,k_i]) -\lambda_2 ([0,l_i]) = \delta
\end{align}
is satisfied. Note that $b_1 = a_1$. If $b_i \leq \delta$, then $[l_j, k_j]$, $j=1,\ldots,i$, and $[0,k_{i+1}]$ cover $[0,1]^2\setminus ([0,1]\times [0,a_1])$, hence
$N_{[ \ ]}([0,1]^2\setminus([0,1]\times [0,a_1]),\delta) \le f = f(\delta) := i+1$. 
Otherwise, $b_{i+1}$ is defined according to \eqref{eq1} and hence given by the equation
\begin{align*} 
	\delta = b_i \cdot 1- b_{i+1} \cdot b_1
\end{align*}
which is equivalent to
\begin{align} \label{eq3}
b_{i+1} = \frac{b_i-\delta}{b_1}.
\end{align}
From equation~\eqref{eq3} it follows inductively that
\begin{align} \label{eq4}
	b_{i+1} = b_1^{-i+1} - \delta \sum_{j=1}^i b_1^{-j} = b_1^{-i+1} - \delta \cdot b_1^{-1} \cdot \frac{b_1^{-i}-1}{b_1^{-1}-1}.
\end{align}
Now the aim is to calculate the minimal $i$ such that $b_i \le \delta$ -- recall that this $i$ satisfies $f=f(\delta) = i+1$.
From \eqref{eq4} it follows that the latter condition is equivalent to
 $$b_1 \cdot \frac{1-b_1}{1-b_1^i} \leq \delta,$$ 
which yields
$$i \geq \frac{\log(1-\delta^{-1}(b_1-b_1^2))}{\log(b_1)}.$$
Plugging in $b_1 = \sqrt{1-\delta}$ leads to
\begin{align*}
f  & = \left\lceil \frac{\log\left( 1 - \delta^{-1}(\sqrt{1-\delta}-(1-\delta)) \right)}{\log\left(\sqrt{1-\delta}\right)} \right\rceil + 1\\
& = \left\lceil 2 \frac{\log\left( 1 - \sqrt{1-\delta} \right) - \log(\delta)}{\log(1-\delta)} \right\rceil + 1.
\end{align*}
Applying the estimate $\sqrt{1-x} \leq 1 - \tfrac{x}{2}$ for $0 \leq x \leq 1$, we obtain
\begin{align*}
f \leq \left\lceil \frac{-2\log(2)}{\log(1-\delta)} \right\rceil + 1
\end{align*}
Clearly, we can cover $[0,1]^2 \setminus ([0, a_1] \times [0,1])$ with $f$ different $\delta$-brackets by simply applying the reflection $\Psi(x_1,x_2) = (x_2,x_1)$ to the $\delta$-brackets we used to cover  $[0,1]^2 \setminus ([0, 1] \times [0, a_1])$. Consequently, we can cover $[0,1]^2 \setminus [0, \underline{a}_1]$ with $2f-1$ different $\delta$-brackets. This implies the claim for $q=1$. For $q > 1$, note that the transformation $\Phi(x_1,x_2) = a_{q-1}^{-1}\cdot (x_1,x_2)$ maps $[0,\underline{a}_{q-1}] \setminus [0,\underline{a}_{q}]$ onto $[0,1]^2 \setminus ([0,\sqrt{1-\delta_q}] \times [0,\sqrt{1-\delta_q})]$ and the general result thus follows from the case $q=1$.
\end{proof}
Now we use \eqref{ineq:brack_decomposition} and the preparatory Lemma~\ref{lem:elementary_brackets} 
to improve the bound for the bracketing number in dimension $d=2$.

\begin{theorem} \label{thm:brack2} Let $0 < \delta < 1$ be arbitrary. For the bracketing number of $[0,1]^2$ we have
	\begin{align}\label{bound_d=2}
	N_{[\ ]}(2, \delta) 
	\leq 2 \log(2) \delta^{-2} + 3(\log(2) + 1) \delta^{-1} - \left( \frac{13}{9} \log(2) -1 \right).
	\end{align}
\end{theorem}

\begin{proof} Let $\kappa_i := 2f(\delta_i) - 1$, where $\delta_i$ is defined as in Lemma~\ref{lem:elementary_brackets}. Then \eqref{ineq:brack_decomposition} and Lemma~\ref{lem:elementary_brackets} imply
	$$N_{[\ ]}(2, \delta) \leq \sum_{i=1}^n \kappa_i \leq \sum_{i=1}^{n-1} 4 \log(2) \frac{1}{|\log\left(1 - \delta_i \right)|} + 3(n-1) + 1.$$
Recall that $|\log(1-x)| \geq x + \tfrac{x^2}{2} + \tfrac{x^3}{3}$ for $0\le x < 1$. Hence it follows that
\begin{align*}
	N_{[\ ]}(2, \delta) & \leq 4 \log(2) \sum_{i=1}^{n-1} \frac{1}{\delta_i \left( 1 + \frac{\delta_i}{2} + \frac{\delta_i^2}{3}\right)}+3(n-1) + 1\\
	& \leq  4\log(2) \sum_{i=1}^{n-1} \frac{\delta_i^{-1}}{1+\frac{\delta}{2}+\frac{\delta^2}{3}} +3(n-1)+1\\
	& = 4\log(2) \sum_{i=1}^{n-1} \frac{\delta^{-1}-i+1}{1+\frac{\delta}{2}+\frac{\delta^2}{3}} +3(n-1)+1.
\end{align*}
From $(n-1)(\delta^{-1}+1) - \tfrac{n(n-1)}{2} \leq \tfrac{1}{2}\delta^{-2} + \delta^{-1}$ we finally deduce
 \begin{align*}
 	N_{[\ ]}&(2, \delta) \leq 2\log(2) \left(1 + \tfrac{\delta}{2}+\tfrac{\delta^2}{3} \right)^{-1} (\delta^{-2}+2\delta^{-1}) + 3(n-1)+1\\
 	& = 2\log(2)  \left(1 + \tfrac{\delta}{2}+\tfrac{\delta^2}{3} \right)^{-1} \left( \delta^{-2}  \left(1 + \tfrac{\delta}{2}+\tfrac{\delta^2}{3} \right) +\frac{3}{2} \delta^{-1} \left(1 + \tfrac{\delta}{2} \right) - \tfrac{13}{12} \right) + 3(n-1) + 1\\
 	& \leq 2 \log(2) \delta^{-2} + 3(\log(2)+1) \delta^{-1} - \left(\frac{13}{9}\log(2)-1\right).
 \end{align*}
\end{proof}

\begin{remark}
It should be mentioned that in \cite{EJoC08} the following constructive upper bound for the bracketing number in dimension $d=2$ was derived:
\begin{align}\label{ejoc08}
	N_{[\ ]}(2, \delta) = \delta^{-2} + o(\delta^{-2}).
	\end{align}
This bound is asymptotically better than the one we presented in Theorem~\ref{thm:brack2}, and it is (essentially) optimal as was shown in \cite{Gne08} by the following lower bound
\begin{align*}
	N_{[\ ]}(2, \delta) = \delta^{-2} - O (\delta^{-1}).
	\end{align*}
The essential point in the constructive upper bound in Theorem~\ref{thm:brack2} is that we have control over the coefficients of all powers of $\delta^{-1}$; we even know the exact number of brackets of the underlying $\delta$-bracketing cover. This is not the case for the construction that leads to \eqref{ejoc08}. It seems to be too involved to provide a simple formula for the exact number of its $\delta$-brackets.
\end{remark}

\subsection{Upper Bound on the Bracketing Number for Arbitrary Dimension} \label{SubSec:2.2}

\paragraph{Faulhaber's formula.} An important ingredient for our improved upper bound on the bracketing number for arbitrary dimensions is a generalized version of the so-called Faulhaber formula, see \cite{Knu93}. Using the whole Faulhaber formula directly would be of limited use in our context because it also involves negative coefficients. Furthermore in our calculations for improving the bound on the bracketing numbers, there do not appear sums of the form $\sum_{i=1}^n i^j$ but sums involving a shift parameter $0 < r < 1$, i.e., $\sum_{i=1}^n (i+r)^j$. Therefore, we derive at first the following inequality.

\begin{theorem}[Generalized Faulhaber inequality] \label{prop:Faulhaber} 
Let $j,n \in \N$ and $0 \le r \le 1$ be arbitrary. Then the following inequality holds
\begin{equation}\label{generalized_faulhaber_inequality}
\sum_{i=1}^n (i+r)^j \leq \frac{(n+r)^{j+1}}{j+1} + \frac{(n+r)^j}{2} + \frac{j(n+r)^{j-1}}{12}.
\end{equation}
\end{theorem}

\begin{proof} We prove the assertion by induction on $n$. For $n=1$ we have $(1+r)^j$ on the left hand side of \eqref{generalized_faulhaber_inequality} and 
	$$(1+r)^j \underbrace{\left( \frac{1+r}{j+1}+\frac{1}{2} + \frac{j}{12(1+r)}\right)}_{=:f_r(j)}$$
on the right hand side of \eqref{generalized_faulhaber_inequality}.
Allowing arbitrary $x\in (-1,\infty)$, we see that $f_r(x)$ has a unique global minimum in 
$x_r := \sqrt{12}(1+r)-1$. This leads for fixed $r\in [0,1]$ to the minimal value
$$f_r(x_r) = \frac{2}{\sqrt{12}} + \frac{1}{2} - \frac{1}{12(1+r)}.$$
For $r \ge 0.07736$ we have $f_r(x_r) >1$. For $r\in [0, 0.07736)$ the minimum of $f_r(j)$, $j\in\N$, is taken in $j = \lfloor x_r \rfloor =2$ or in $j = \lceil x_r \rceil =3$. For both values of $j$ we have $f_r(j) \ge 1$, since $f_0(j) = 1$ and $\frac{\partial}{\partial r} f_r(j) \ge 0$.\\[12pt]
Now we consider $n+1$. Applying the induction hypothesis and binomial expansion to the left hand side of \eqref{generalized_faulhaber_inequality} yields
\begin{align*}
	\sum_{i=1}^{n+1} (i+r)^j & \le \frac{(n+r)^{j+1}}{j+1} + \frac{(n+r)^j}{2} + \frac{j(n+r)^{j-1}}{12} + (n+1+r)^j\\
	&= \frac{(n+r)^{j+1}}{j+1} + \frac{(n+r)^j}{2} + \frac{j(n+r)^{j-1}}{12} + \sum_{k=0}^j \binom{j}{k} (n+r)^k.
\end{align*} 
Similarly, the right hand side of \eqref{generalized_faulhaber_inequality} can be written as
\begin{align*}
	\frac{1}{j+1} \sum_{k=0}^{j+1} \binom{j+1}{k}(n+r)^k + \frac{1}{2} \sum_{k=0}^j \binom{j}{k} (n+r)^k + \frac{j}{12} \sum_{k=0}^{j-1} \binom{j-1}{k}(n+r)^k.
\end{align*}
Next we compare the coefficients of the powers of $n+r$ on both sides. For $k=j+1,j,j-1,j-2,j-3$ the coefficients are equal on both sides, namely  $\tfrac{1}{j+1}, \tfrac{3}{2}, \tfrac{13j}{12}, \tfrac{j(j-1)}{2}$ and $\tfrac{j(j-1)(j-2)}{6}$. Let $k=j-s \in \N_0$ with $j \geq s \geq 4$ be arbitrary. Then the coefficient on the left side is $\binom{j}{j-s}$. On the right side, the coefficient is
\begin{align*}
	\frac{1}{j+1} \binom{j+1}{j-s} + \frac{1}{2} \binom{j}{j-s} + \frac{j}{12} \binom{j-1}{j-s} = \left(\frac{1}{s+1} + \frac{1}{2} + \frac{s}{12} \right) \binom{j}{j-s}.  
\end{align*}
The expression in the bracket is at least $1$ for $s \geq 3$. This finishes the proof.
\end{proof}

\paragraph{Main result on bracketing numbers.} The generalized Faulhaber inequality from Theorem~\ref{prop:Faulhaber} is the main tool for improving the upper bound on the bracketing number for arbitrary dimension.

\begin{theorem} \label{thm:general_bracketing_number} 
For all $0 < \delta < 1$ and $d \in \mathbb{N}$ we have
	\begin{equation} \label{ineq:brack}
	N_{[\ ]}(d, \delta) \le \max\left({1.1}^{d-101},1\right) \frac{d^d}{d!} (\delta^{-1} + 1)^d.
	\end{equation}
\end{theorem}

In order to keep the presentation of the proof of Theorem~\ref{thm:general_bracketing_number} as clear as possible, we present two preparatory results as separate statements. 

\begin{lemma}\label{Lemma:Maximum_1} 
Given $d \in \mathbb{N}$, the function
	$$f(k) = \frac{d!}{d^{k-1} (d-k+1)!} \left( \frac{k}{12} + \frac{1}{2} \right)$$
	with $k \in \mathbb{N}$ is maximal for $k = -3 + { \lceil \sqrt{16+d} \rceil}$.
\end{lemma}

\begin{proof}
	We have
\begin{align*}
	\frac{f(k+1)}{f(k)} & = \frac{k+7}{k+6} \cdot \frac{d-k+1}{d}.
\end{align*}
	This expression is ${\le} 1$ if and only if
\begin{align*}
	\frac{k+7}{k+6} (d-k+1) &= \frac{k+7}{k+6} d + \frac{(k+7)(1-k)}{k+6} {\le} d.
\end{align*}
which is yet equivalent to
\begin{align*}
	- k^2 - 6k + 7 + d & {\le} 0
\end{align*}
or, in other words,
\begin{align*}
	k &{ \ge} -3 + \sqrt{16+d}.
\end{align*}
This implies the claim.
\end{proof}

\begin{lemma} \label{cor:max}
	Given $d \in \mathbb{N}$, the function
	$$g_d(k) = \frac{d!}{d^{k-1} (d-k+1)!}  \left( \frac{k}{12} + \frac{1}{2} + \frac{1}{k+1} \right) $$
	with $k \in \mathbb{N}$ takes its maximum in  $k_{\rm max} \le \min\{  -3 + {\lceil \sqrt{16+d} \rceil}\,,\, 1 + { \lceil 0.0544 \cdot d \rceil} \}$.
\end{lemma}

\begin{proof}
Since $r(k) := g_d(k)-f(k)$ is strictly monotone decreasing on $\N$, we obtain from Lemma~\ref{Lemma:Maximum_1} that  $k_{\rm max} \le  -3 + {\lceil \sqrt{16+d} \rceil}$, because $g(k) = f(k) + r(k)$ and both $f(k_1) < f( -3 + {\lceil \sqrt{16+d} \rceil})$ and $r(k_1) < r( -3 + {\lceil \sqrt{16+d} \rceil})$ for all $k_1 >  -3 + {\lceil \sqrt{16+d} \rceil}$. Furthermore, we have for all $1 \le k < d$
$$
h(k) := \frac{g_d(k+1)}{g_d(k)} = \frac{d-k+1}{d} \cdot \frac{k+1}{k+2} \cdot \frac{k^2 + 9k + 26}{k^2 + 7k + 18},
$$
and 
$$
u(k) := \frac{d}{d-k+1} h(k)
$$ 
satisfies $\lim_{k\to \infty} u(k) =1$. It is easily checked that $u(k)$ takes its maximum in $k=7$, for $k \in \mathbb{N}$, namely 
$u(7)= 1.05747126.$ Hence, if $k_0 = k_0(d) \in \N$ satisfies the estimate
$$\frac{d-k_0+1}{d} \cdot 1.05747126 \le 1,$$
then all $k\ge k_0$ satisfy it, too, implying $h(k)\le 1$ and thus $g_d(k) \ge g_d(k+1)$. This yields the necessary condition 
$$k_{\rm max} \le 1 + {\lceil d \cdot 0.0544 \rceil}.$$
\end{proof}

\begin{proof}[Proof of Theorem~\ref{thm:general_bracketing_number}] Let $b_d := \max\left(1.1^{d-101},1\right)$ and $n:= \lceil \delta^{-1} \rceil$.  We proceed by induction over the dimension $d$. For $d=1$ the claim is obviously true. Let now $d\ge 2$.
Analogously to the proof of Theorem 1.15 in \cite{Gne08} the induction hypothesis with induction step $d-1 \to d$ yields
	\begin{align*}
		N_{[\ ]}(d, \delta) & \le \sum_{k=1}^{d-1} \binom{d}{k} b_{d-k} \frac{d^{d-k}}{(d-k)!} \left( \sum_{i=1}^{n-1} \left(\delta^{-1} - i + 1 \right)^{d-k} \right) + \delta^{-1} + 1
	\end{align*}
	(where in \cite{Gne08} $b_\ell$ was actually equal to $2^\ell$ for all $\ell \in \N$).
Applying the generalized Faulhaber inequality from Theorem~\ref{prop:Faulhaber} with $j = d-k$ and $r = \delta^{-1} - n + 1$ we get
	\begin{equation}\label{threestar}
	\begin{split}
		N_{[\ ]}(d, \delta) \le \sum_{k=1}^{d-1} \binom{d}{k} &b_{d-k} \frac{d^{d-k}}{(d-k)!} \left( \frac{1}{d-k+1} (\delta^{-1})^{d-k+1} + \frac{1}{2} (\delta^{-1})^{d-k} + \frac{d-k}{12} (\delta^{-1})^{d-k-1} \right)\\
		& + \delta^{-1} + 1.
\end{split}
\end{equation}
The right hand side of inequality~\eqref{ineq:brack} is equal to
\begin{align}\label{fourstar}
 	b_d \frac{d^d}{d!} (\delta^{-1} + 1)^d = \sum_{k=0}^d b_d \binom{d}{k} \frac{d^d}{d!} (\delta^{-1})^{d-k} = \sum_{k=0}^d b_d  \frac{d^d}{k!(d-k)!} (\delta^{-1})^{d-k}.
\end{align}
Now we compare the coefficients in front of the powers $p$ of $\delta^{-1}$ in \eqref{threestar} and \eqref{fourstar}. We start with the lowest powers, $p=0$ and $p=1$.
For $p=0$ the coefficient in \eqref{threestar} is 
$${d \choose d-1} b_1 \frac{d}{12} + 1 = \frac{d^2}{12} + 1$$
and the corresponding coefficient in \eqref{fourstar} is $b_d \frac{d^d}{d !}$.
It is easily checked that
$$ \frac{d^2}{12} + 1  \le \frac{d^d}{d !} \le b_d \frac{d^d}{d !} \hspace{3ex}\text{for all $d\ge 2$.}$$
For $p=1$ the coefficient in \eqref{threestar} is
$$\frac{d^2}{2} + \frac{d^3(d-1)}{24} + 1$$
and the corresponding coefficient in \eqref{fourstar} is $b_d \frac{d^d}{(d-1)! }$, and again it is easily verified that 
$$\frac{d^2}{2} + \frac{d^3(d-1)}{24} + 1 \le \frac{d^d}{(d-1)! } \le b_d \frac{d^d}{(d-1)! }
\hspace{3ex}\text{for all $d\ge 2$.}$$
For $p = d$ the corresponding coefficients in \eqref{threestar} 
is not larger than the one in \eqref{fourstar}, since $b_{d-1}\le b_d$. 
For $p= d-1$, $p\ge 2$, it can be easily checked that the coefficient in \eqref{threestar} is strictly less than the one in \eqref{fourstar}. Next we compare the coefficients for powers $2 \le p \le d-2$, i.e., for powers $d-k$ with $2 \le k \le d- 2$.   {In \eqref{threestar} the respective coefficient is given by}
\begin{align*}
	& \frac{d^{d-k-1}}{(d-k)!} \left( b_{d-k-1} \binom{d}{k+1} + b_{d-k} \frac{d}{2} \binom{d}{k} + b_{d-k+1} \frac{d^2}{12} \binom{d}{k-1} \right)\\
	&  \leq \frac{d^{d-k-1}}{(d-k)!}  b_{d-k+1} \underbrace{\left(  \binom{d}{k+1} + \frac{d}{2} \binom{d}{k} + \frac{d^2}{12} \binom{d}{k-1} \right)}_{=:a_{d-k}}.
\end{align*}
{To show that this  coefficient is smaller than the corresponding one from \eqref{fourstar} it suffices to show that 
\begin{equation}\label{hinreichende_ungleichung}
\frac{b_{d-k+1}}{b_d} \frac{k!}{d^{k+1}} a_{d-k} \le 1.
\end{equation}
}
{ At first, let us consider the case $d \leq 101$, i.e. $b_d = 1= b_{d-k+1}$. 
We  obtain} the inequality
\begin{equation}\label{twostar}
\begin{split}
	A_{d-k} := k! \cdot \frac{a_{d-k}}{d^{k+1}}  =& \frac{d!}{d^{k-1}(d-k+1)!} \cdot \frac{1}{d^2} \left( \frac{(d-k+1)\cdot(d-k)}{k+1} + \frac{d\cdot(d-k+1)}{2} + \frac{d^2 \cdot k}{12} \right)\\
	<& \frac{d!}{d^{k-1}(d-k+1)!} 
	 \underbrace{\left( \frac{1}{k+1} + \frac{1}{2} + \frac{k}{12} \right)}_{{= f_0(k)}} = g_d(k) \leq g_d(k_{\rm max}) {\le f_0(k_{\max}),}
\end{split}
\end{equation}
{where $k_{\max}$ is as in  Lemma~\ref{cor:max}.}
{From this and Lemma~\ref{cor:max} we get for $d \leq 36$ that $A_{d-k} \le 1$,} because
 $k_{\rm max} \le \min\{  -3 + \lceil \sqrt{16+d} \rceil\,,\, 1 + \lceil 0.0544 \cdot d \rceil \} <4$ {and  hence $f_0(k_{\max}) \le 1$}.
It can be checked by a simple computer calculation of all possible values of $A_{d-k}$ that in fact $A_{d-k} \leq 1$ holds for $d \leq 101$. Thus {\eqref{hinreichende_ungleichung} holds and consequently} the coefficients in \eqref{threestar} are smaller than those in \eqref{fourstar} for all powers of $\delta^{-1}$ and the claim follows for $d \leq 101$. Now let $d > 101$. For all $k$ with $d-k+1 > 101$ we see that 
$$ \frac{b_{d-k+1}}{b_d} = (1.1)^{-k+1}.$$ 
Thus $\frac{b_{d-k+1}}{b_d} \leq \frac{1}{k}$ for $k \geq 40$ 
implying
$$ {\frac{b_{d-k+1}}{b_d} A_{d-k} \le}
 \frac{b_{d-k+1}}{b_d} {f_0(k)}  \leq \frac{1}{k} {f_0(k)} < 1$$
in these cases. Moreover the inequality
$$ \frac{b_{d-k+1}}{b_d} {f_0(k)} < 1$$
can be checked in the remaining cases $k=2,\ldots, 39$ by hand. Hence, {due to \eqref{hinreichende_ungleichung},} the corresponding coefficients in \eqref{threestar} are smaller than those in \eqref{fourstar}. If $d - k + 1 \leq 101$, then we get from \eqref{twostar}
\begin{align*} 
	\frac{b_{d-k+1}}{b_d} A_{d-k} {=  \frac{1}{1.1^{d-101}}
	 A_{d-k} \le   \frac{1}{1.1^{d-101}} g_d(k_{\max}) \le  \frac{1}{1.1^{d-101}} f_0(k_{\max}) =: \tilde{g}_d(k_{\max})}
\end{align*}
for all {$d-100 \le k \le d-2$}. A direct calculation yields {$k_{\rm max}(102) = 7$ and $g_{102}(7) < 1.01$}. Hence \eqref{hinreichende_ungleichung} is true for $d=102$. 
{Note that $f_0(k)$ is monotonic increasing for $k\ge 2$ and, due to
Lemma~\ref{cor:max}, 
$$k_{\max} = k_{\rm max}(d) \le \min\{-3 + \lceil \sqrt{16+d} \rceil, 1 + \lceil 0.0544 \cdot d \rceil \} =:k_0(d)$$
which is monotonic increasing as well, we have that
$$\frac{b_{d-k+1}}{b_d} A_{d-k} \leq \tilde{g}_d(k_0(d)).$$
Now $k_0(103) = 7$ implies
 $$\tilde{g}_{103}(k_{\max}(103)) \le \tilde{g}_{103}(k_{0}(103)) < 0.999.$$
Since $f_0(k+1)/f_0(k) < 1.1$ for $k\ge 7$ and 
$|k_0(d+1) - k_0(d)| \le 1$ for all $d$, we see that $\tilde{g}_d(k_{\max}(d)) \le 1$ for all $d\ge 103$.
This establishes \eqref{hinreichende_ungleichung} in the case $d > 101$ and finishes the proof.}
\end{proof}

\begin{remark} Note that our argumentation would fail when using the complete Faulhaber formula instead of the inequality from Theorem~\ref{prop:Faulhaber} because there appear negative terms in it. \end{remark}

\section{Application to Discrepancy Bounds}

In this section, we apply the new general upper bound on the bracketing number from Theorem~\ref{thm:general_bracketing_number} to derive new discrepancy bounds for the smallest achievable star-discrepancy and weighted star-discrepancy with explicitly given dependence on the number of points $N$ and on the dimension $d$. We will not restrict our analysis to classical Monte Carlo point sets but use the more general notion of negatively dependent sampling which we introduce first.

\subsection{Negative Dependence of Sampling Schemes}\label{SUBSEC:Neg_Dep} 

Recall that classical Monte Carlo involves random samples which are independent and identically uniformly distributed on $[0,1)^d$. Monte Carlo methods do not yield deterministic convergence and hence also result in random errors and the variance of the estimation has to be taken into account. Therefore, several variance reduction techniques for Monte Carlo have been established, see, e.g., \cite{Gla03}, Chapter~4. For this and other similar reasons, researchers proposed to use randomized Quasi-Monte Carlo methods, see, e.g., \cite{CP76}, \cite{Owe95} or \cite{LEc17}. A very recent research direction is to use notions of negative dependence to analyze random points sets, see, e.g., \cite{GH16}, \cite{Lem18}, \cite{WG19}, \cite{WGH19}, \cite{WLD19}. 
This line of research is motivated in detail in the introductions of the papers \cite{DDG18, GH16, Lem18}.
We consider here the notions of negative dependence studied in \cite{GH16} and \cite{WGH19}.

\begin{definition} Let $\gamma \geq 1$. The binary random variables $T_1, T_2, \ldots, T_N$ are called \textbf{upper $\gamma$-negatively dependent} if
	$$\mathds{P} \left( \bigcap_{j \in u} \left\{ T_j = 1 \right\} \right) \leq \gamma \prod_{j \in u} \mathds{P} \left( T_j = 1 \right), \quad \textrm{for all } u \subset [N]$$
	and  \textbf{lower $\gamma$-negatively dependent} if
	$$\mathds{P} \left( \bigcap_{j \in u} \left\{ T_j = 0 \right\} \right) \leq \gamma \prod_{j \in u} \mathds{P} \left( T_j = 0 \right), \quad \textrm{for all } u \subset [N],$$
	where the set $\left\{1, 2, \ldots,N \right\}$ is denoted by $[N]$.
	If both conditions are satisfied, then $T_1, T_2, \ldots, T_N$ are called \textbf{$\gamma$-negatively dependent}.
\end{definition}

In this paper, we are interested in binary random variables $T_i, i=1,\ldots,N$, of the form $T_i = \mathds{1}_A(X_i)$, where $A$ is a Lebesgue-measurable subset of $[0,1]^d$, $\mathds{1}_A$ denotes the charactaristic function of $A$, and $X_1,\ldots,X_N$ are randomly chosen points in $[0,1]^d$. If such a sequence is $\gamma$-negatively dependent, then it fulfills special bounds of Hoeffding- and Bernstein-type, respectively, see e.g. \cite{Heb12}.
\begin{definition} A randomized point set $P=(p_j)_{j=1}^N$ in $[0,1)^d$ is called a \textbf{sampling scheme}, if every single $p \in P$ is distributed uniformly in $[0,1)^d$ and the vector $(p_1,\ldots,p_N)$ is exchangeable, meaning that for any permutation $\pi$ on $[N]$ it holds that the probability distribution of $(p_1,\ldots,p_N)$ is the same as the one of $(p_{\pi(1)},\ldots,p_{\pi(N)})$.
\end{definition}

The assumption of exchangeability is easily satisfied.
Let $\tilde{P}$ be an $N$-point set such that every $\tilde{p} \in \tilde{P}$ is uniformly distributed in $[0,1)^d$ and let $\pi$ be a random uniformly chosen permutation of $[N]$. Then $P:= (\tilde{p}_{\pi(j)})_{j=1}^n$ is already a sampling scheme. 

\begin{definition} For a given set $\mathcal{S}$ of Lebesgue measurable subsets of $[0,1)^d$ we say that a sampling scheme $P=(p_j)_{j=1}^N$ is \textbf{$\mathcal{S}$-$\gamma$-negatively dependent} if for every $Q \in \mathcal{S}$ the random variables $(\mathds{1}_Q(p_j))_{j=1}^N$ are $\gamma$-negatively-dependent.
\end{definition}
In this paper, we are interested in two special cases, namely $\mathcal{S} \in \left\{ \mathcal{C}_0^d, \mathcal{D}_0^{d} \right\}$, where
$$\mathcal{C}_0^d := \left\{ [0,a) : a \in [0,1)^d \right\}$$
is the set of boxes anchored at $0$ (sometimes called \textit{corners}) and 
$$\mathcal{D}_0^{d}  := \left\{ Q \setminus R : Q,R \in \mathcal{C}_0^d \right\}$$
is the set of differences of boxes from $\mathcal{C}_0^d$.

\subsection{Bounds for the Star-Discrepancy}\label{SUBSEC:Star_Discrepancy}

For deriving discrepancy bounds, we will use an intermediate result of the proof of Theorem~4.4 in \cite{GH16} and therefore stick to the notation therein. Hence, we need to introduce some variables. First let $\mu \in \mathbb{N}, \mu \geq 2$ be arbitrary and 
$$c_\mu := \frac{1}{1-\sqrt{\frac{\mu+1}{2\mu}}}.$$
Furthermore let $c_0, c_1 > 0$ be two constants and choose $K\ge \mu$ as the smallest natural number with
$$\frac{1}{\sqrt{K2^K}} \leq c_0 c_1 c_\mu \sqrt{\frac{d}{N}}$$
and set
$$\tau_\mu := \frac{c_1^2}{4(1+1/(3c_\mu))} \quad \text{implying} \quad c_1 = \sqrt{4\tau_{\mu}\left(1+\frac{1}{3c_{\mu}} \right)}.$$
Note that $c_0$ may still be chosen arbitrarily. For each $\mu \leq k \leq K$ we choose a $2^{-k}$-cover $\Gamma_k$ of minimum size.

\begin{lemma}[\cite{GH16}, Proof of Theorem 4.4] \label{lem:GH} Let $d, N \in \mathbb{N}$ and $\rho \in [0,\infty)$. Let $X=(X_j)_{j=1}^N$ be a $\mathcal{D}_0^d$-$e^{\rho d}$-negatively dependent sampling scheme. Then there exist events $E_k$, $k= \mu, \ldots, K$, with 
		\begin{align} \label{ineq:1}
	\mathds{P}(E_\mu^c) \leq 2 |\Gamma_\mu|e^{\rho d} e^{-2c_0^2d}
	\end{align}
	and
	\begin{align} \label{ineq:2}
	\mathds{P}(E_k^c) \leq 2 |\Gamma_{k}|e^{\rho d} e^{-2c_0^2\tau_\mu(k-1)d}
	\end{align}
	for $\mu < k \le K$ and if the event
	$$E := \bigcap_{k=\mu}^K E_k$$
         occurs then any realization $P$ of $X$ satisfies
	$$D_N^*(P) \leq c_0 \left( 1 + c_1c_\mu \sqrt{\frac{\mu}{2^\mu}} \right) \sqrt{\frac{d}{N}}.$$	
\end{lemma}

From Lemma~\ref{lem:GH} and Theorem~\ref{thm:general_bracketing_number} it follows that the star-discrepancy of a negatively dependent sampling scheme fulfills
$$D_N^*(X) \leq c \sqrt{\frac{d}{N}}$$
with high probability as long as $c$ is sufficiently large.

\begin{theorem} \label{thm:improved_star_discrepancy_bound} 
Let $d, N \in \mathbb{N}$ with $d \geq 2$ and $\rho \in [0,\infty)$. Let $X=(X_j)_{j=1}^N$ be a $\mathcal{D}_0^d$-$e^{\rho d}$-negatively dependent sampling scheme. Then for every $c > 2.4968$
\begin{equation}\label{est:disc*}
	D_N^*(X) \leq c \sqrt{\frac{d}{N}}
\end{equation}
	holds with probability at least $1 - e^{-(1.67681c^2-10.45292-\rho) \cdot d}$ implying that for every $q \in (0,1)$
	\begin{align} \label{ineq:star_discrepancy}
	D_N^*(X) \leq 0.7723 \sqrt{10.45292 + \rho + \frac{\log((1-q)^{-1})}{d}} \sqrt{\frac{d}{N}}
	\end{align}
	holds with probability at least $q$.
\end{theorem}

\begin{proof} We consider the sets $E_\mu, \ldots, E_K$, and $E$ from Lemma~\ref{lem:GH}. Then Theorem~\ref{thm:general_bracketing_number} and Sterling's formula imply

	$$|\Gamma_k| \leq 2\frac{1}{\sqrt{2 \pi d}} b_d e^d (2^k+1)^d$$
	with $b_d := \max \left( 1 , 1.1^{d-101} \right)$. Defining $\sigma := \mu - \log(2^\mu+1) - 1 - \frac{\log(b_d)}{d}-\tfrac{\log(2)}{d}$ and \eqref{ineq:1} yield
	$$\mathds{P}(E_\mu^c) \leq \sqrt{\frac{2}{\pi d}} e^{-(2c_0^2-\rho-\mu+\sigma)d}.$$ 
	Similarly, we get from \eqref{ineq:2} the inequality
	\begin{align*}
	\mathds{P}(E_k^c) & \leq 2\sqrt{\frac{2}{\pi d}} b_d e^d 2^{kd} \left(1+2^{-k}\right)^d e^{\rho d}\exp\left(-2c_0^2 \tau_\mu (k-1)d\right) \\
		& \leq \sqrt{\frac{2}{\pi d}} e^{\left(1 + \zeta + \rho \right)d}\exp\left(-(2c_0^2\tau_\mu-\log(2))(k-1)d\right)
	\end{align*}
	for $\mu + 1 \leq k \leq K$, where $\zeta = \log(1+2^{-\mu-1}) + \log(2)+ \frac{\log(2)}{d} + \tfrac{\log(b_d)}{d}$.
Using $\mathds{P}(E) = 1 - \mathds{P}(E^c) \geq 1 - \sum_{k=\mu}^K \mathds{P}(E_k^c)$ it follows that
	\begin{align*}
		& \mathds{P}(E) \geq 1 -\sqrt{\frac{2}{\pi d}} \left( e^{-(2c_0^2-\rho-\mu+\sigma)d} \right.\\
		& \qquad \left. +  e^{\left(1 + \zeta + \rho \right)d} e^{\left(-(2c_0^2\tau_\mu-\log(2))\mu d\right)} \sum_{j=0}^{K-\mu-1} e^{\left(-(2c_0^2\tau_\mu-\log(2))j d\right)} \right)\\
		& \geq 1 - \sqrt{\frac{2}{\pi d}} e^{-(2c_0^2-\rho-\mu+\sigma)d} \left( 1 + \frac{e^{-(2c_0^2(\mu\tau_\mu-1)+(1-\log(2))\mu-1-\zeta-\sigma)d}}{1 - e^{\left(-(2c_0^2\tau_\mu-\log(2)) \right) d} } \right)\,.
\end{align*}		
 If $$c_0 \geq \sqrt{(\mu+\rho-\sigma)/2},$$ then we obtain		
	\begin{align*}
		 \mathds{P}(E) 
		 \geq 1 - \sqrt{\frac{2}{\pi d}} e^{-(2c_0^2-\rho-\mu+\sigma)d} \left( 1 + \frac{e^{-((\mu+\rho-\sigma)(\mu\tau_\mu-1)+(1-\log(2))\mu-1-\zeta-\sigma)d}}{1 - e^{\left(-((\mu+\rho-\sigma)\tau_\mu-\log(2)) \right) d} } \right).
	\end{align*}
	We find that 
	\begin{align} \label{eq2}
	1 + \frac{e^{-((\mu+\rho-\sigma)(\mu\tau_\mu-1)+(1-\log(2))\mu-1-\zeta-\sigma)d}}{1 - e^{\left(-((\mu+\rho-\sigma)\tau_\mu-\log(2)) \right) d} } < \sqrt{\frac{\pi d}{2}}
	\end{align}
	holds for $\mu= 12$ and $\tau_\mu = 0.0871$. This can be checked for $d=2$ and $\rho = 0$ by a computer calculation. As the left hand side of \eqref{eq2} is monotonic decreasing in $d$ and in $\rho$ and the right hand side is monotonic increasing in $d$, the general case follows. By inserting these explicit values into the formulas for $c_\mu$ and $c_1$ and relate $c$ and $c_0$ via
	$$ c = c_0 \left(1 +c_1 c_\mu \sqrt{ \tfrac{\mu}{2^\mu}} \right),$$
	we obtain the explicit lower bound for the probability. The minimal possible value for $c$ is taken if $c_0 = \sqrt{(\mu-\sigma)/2)}$ and hence estimate \eqref{est:disc*} holds only if $c > 2.4968$. Estimate \eqref{ineq:star_discrepancy} follows by putting $q:= 1 - e^{-(1.67681c^2-10.45292-\rho) \cdot d}$ and solving for $c$.
\end{proof}
\begin{remark} Note that also the probability $1 - e^{-(1.67681c^2-10.45292-\rho) \cdot d}$ is positive if $c > 2.4968$.
\end{remark}

In the special case of Monte Carlo points $X_1, X_2, \ldots, X_N$ we have $\rho = 0$ and our bound~\eqref{ineq:star_discrepancy} improves the result from \cite{AH14}, Theorem~1, and \cite{GH16}, Theorem~4.4. Furthermore, the smaller constant in \eqref{ineq:star_discrepancy2} (in comparison to $2.5287$ from \cite{GH16}, Corollary 4.5) can immediately be deduced from Theorem~\ref{thm:improved_star_discrepancy_bound}. Corollary~\ref{cor:improvedC} improves on the result from \cite{AH14}, table at the bottom of p.~1374, and \cite{GH16}, Corollary~4.4.

\begin{corollary}  \label{cor:improvedC} 
Let $d, N \in \mathbb{N}$.  Then there exists a point set $P \subset [0,1)^d$ of $X$ such that
\begin{align} \label{ineq:star_discrepancy2}	
	D_N^*(P) \leq 2.4968 \sqrt{\frac{d}{N}}.
\end{align}
Let $d\ge 2$ and let $X = (X_n)_{n=1}^N$ be a Monte Carlo point set uniformly distributed in $[0,1)^d$.The probability that a realization $P$ of $X$ satisfies
	$$D_N^*(P) \leq 2.5 \sqrt{\frac{d}{N}}, \qquad and \qquad D_N^*(P) \leq 3 \sqrt{\frac{d}{N}}$$
is at least $0.0528$ and $0.9999$, respectively. 
\end{corollary}

\begin{proof}
In the case $d = 1$ the set $P := \{1/2N, 3/2N,\ldots, (2N-1)/2N\}$ satisfies $D^*_N(P) = 1/2N$, which is even a stronger result than \eqref{ineq:star_discrepancy2}. The case $d\geq 2$ can be solved with Theorem \ref{thm:improved_star_discrepancy_bound}. We know that 
$$D_N^*(X) \leq c \sqrt{\frac{d}{N}}$$
holds with probability $q$ which is at least 
$$1 - e^{-(1.67681c^2-10.45292) \cdot d}, $$
and we want to show that the last term is strictly positive. 
Since the exponential is decreasing in $d$, it suffices to consider the case $d=2$. Solving the inequality for $c$ yields the first estimate, namely 
$$c > \sqrt{\tfrac{10.45292}{1.67681}} \geq 2.49676\ldots .$$
The remaining probabilities can be found if we use $c=2.5$ and $c=3$ respectively, i.e. 
$$q \geq 1 - e^{-(1.67681\cdot 2.5^2-10.45292) \cdot 2} \geq 0.0528$$
and
$$q \geq 1 - e^{-(1.67681\cdot 3^2-10.45292) \cdot 2} \geq 0.9999.$$
\end{proof}

In fact, Corollary~\ref{cor:improvedC} indicates that the realization of a Monte Carlo point set typically has star-discrepancy $D_N^*(P)  \leq c \sqrt{\frac{d}{N}}$. Next we show that this is not only the typical case in terms of probability but also in expectation. We split the corresponding statement into two separate parts.

\begin{lemma} \label{lem:expectation} 
Let $d, N \in \mathbb{N}$ and let $X = (X_n)_{n =1}^N$  be a sampling scheme in $[0,1]^d$. Then
	$$\mathds{E}[D_N^*(X)]  = \left( \int_{0}^{\sqrt{N/d}} \mathds{P}\left( D_N^*(X) \geq \tau \sqrt{d/N} \right) \mathrm{d} \tau \right) \sqrt{d/N}.$$
\end{lemma}  

\begin{proof} The equation 
	$$\mathds{E}[D_N^*(X)] = \int_0^1 \mathds{P}(D_N^*(X) \geq t) \mathrm{d} t$$
	is a standard result from probability theory and an immediate consequence of Fubini's Theorem. The result then follows from transformation of variables.
\end{proof}

\begin{proposition}\label{Prop:Expec_Disc}
Let $X= (X_n)_{n = 1}^N$ be a sampling scheme in $[0,1]^d$ and $\alpha, \beta \geq 0$ such that for all $c>0$
	$$\mathds{P}\left( D_N^*(X) \leq c \sqrt{\frac{d}{N}} \right) \geq 1 - e^{-(\alpha c^2 - \beta)d}.$$
Then 
	\begin{align*}
	\mathds{E}[D_N^*(X)] & \leq \sqrt{\frac{\beta}{\alpha}} \left( 1  + \sqrt{\frac{\pi}{\beta d}} e^{\beta d} \left( \Phi\left(\sqrt{2\alpha N}\right) - \Phi\left(\sqrt{2 \beta d}\right) \right)\right) \sqrt{\frac{d}{N}}\\
		& \leq \sqrt{\frac{\beta}{\alpha}} \left( 1 + \frac{1}{2 \beta d} \right) \sqrt{\frac{d}{N}},
	\end{align*}
	where $\Phi(\cdot)$ denotes the cumulative distribution function of the standard normal distribution.
\end{proposition}

\begin{proof} Lemma~\ref{lem:expectation} implies
	$$\mathds{E}[D_N^*(X)] = \int_{0}^{\sqrt{N/d}} \mathds{P}\left( D_N^*(X) \geq \tau \sqrt{d/N} \right) \textrm{d} \tau \sqrt{d/N}.$$
By assumption it holds that 
	$$\mathds{P}\left( D_N^*(X) \geq \tau \sqrt{d/N} \right)  \leq \min \left( 1, e^{-(\alpha \tau^2 - \beta)d} \right)$$
and $e^{-(\alpha \tau^2 - \beta)d} < 1$ if $\tau > \sqrt{\beta/\alpha}$.
Thus
	\begin{align*}
	\mathds{E}[D_N^*(X)] < \left( \int_0^{\sqrt{\beta/\alpha}} \textrm{d} \tau + \int_{\sqrt{\beta/\alpha}}^{\sqrt{N/d}} e^{-(\alpha \tau^2 - \beta)d}\, \textrm{d}\tau \right)\sqrt{d/N}.
	\end{align*}
The first integral is equal to $\sqrt{\beta/\alpha}$. Applying the transformation $x^2 =2 \alpha \tau^2 d$, the second integral can be evaluated as follows:
\begin{align*}
	\int_{\sqrt{\beta/\alpha}}^{\sqrt{N/d}} e^{-(\alpha \tau^2 - \beta)d} \textrm{d}\tau & = e^{\beta d} 	\int_{\sqrt{\beta/\alpha}}^{\sqrt{N/d}} e^{-\alpha \tau^2 d} \mathrm{d}\tau \\
	& = \frac{e^{\beta d}}{\sqrt{2\alpha d}} \int_{\sqrt{2 \beta d}}^{\sqrt{2 \alpha N}} e^{-\frac{x^2}{2}} \mathrm{d} x\\
	& = \sqrt{\frac{\pi}{\alpha d}} e^{\beta d} \left( \Phi \left( \sqrt{2 \alpha N} \right) - \Phi \left( \sqrt{2 \beta d} \right) \right).
	\end{align*}
Because 
\begin{align*}
	\Phi \left( \sqrt{2 \alpha N} \right) - \Phi \left( \sqrt{2 \beta d} \right) & \leq \frac{1}{\sqrt{2\pi}} \int_{\sqrt{2 \beta d}}^\infty e^{-\frac{x^2}{2}} \,\mathrm{d} x 
	 \le \frac{1}{\sqrt{2\pi}} \int_{\sqrt{2 \beta d}}^\infty  \left( 1 + \frac{1}{x^2} \right) e^{-\frac{x^2}{2}} \,\mathrm{d} x\\
	&=  \frac{1}{\sqrt{2\pi}} \left( -\frac{1}{x} \right)  e^{-\frac{x^2}{2}} \bigg|^\infty_{\sqrt{2\beta d} }
	= \frac{1}{\sqrt{2\pi}} \cdot \frac{1}{\sqrt{2 \beta d}} e^{-\beta d}
\end{align*}
also the second inequality follows.
\end{proof}
In explicit numerical terms, we finally get the following inequalities for the expected value:

\begin{corollary} \label{cor:expectation}
	Let $d, N \in \mathbb{N}$ and let $X = (X_n)_{n = 1}^N$ be a Monte Carlo point set uniformly distributed in $[0,1)^d$. Then 
	$$\mathds{E}[D_N^*(X)] \leq 2.55648 \sqrt{\frac{d}{N}}$$
	for $d = 2$ and
	$$\mathds{E}[D_N^*(X)] \leq 2.53657 \sqrt{\frac{d}{N}}$$
	for all $d \geq 3$. 	
\end{corollary}

\begin{remark} Notice that $\sqrt{\beta/\alpha} = 2.49676...$ and thus the quantity $c_d$ in the inequality  $\mathds{E}[D_N^*(X)] \leq c_d \sqrt{d/N}$ in Proposition~\ref{Prop:Expec_Disc} converges to the same constant $c$ as in Corollary~\ref{cor:improvedC} when $d$ goes to $\infty$.
\end{remark}

From Corollary~\ref{cor:expectation}, Theorem~\ref{thm:improved_star_discrepancy_bound} and
from the estimates \eqref{doe1} and \eqref{doe2} established in \cite{Doe13} we immediately obtain the following corollary.

\begin{corollary}\label{Cor_Concentration}
Let $d, N \in \mathbb{N}$ and let $X = (X_n)_{n = 1}^N$ be a Monte Carlo point set uniformly distributed in $[0,1)^d$. Then there exist constants $K_1, K_2 >0$, independent of $d$ and $N$,  such that 
$$K_1 \sqrt{\frac{d}{N}} \le \mathds{E}[D_N^*(X)] \leq K_2 \sqrt{\frac{d}{N}}$$
and the estimate
$$K_1 \sqrt{\frac{d}{N}} \le D^*_N(X) \leq K_2 \sqrt{\frac{d}{N}}$$
holds with probability at least $1-2e^{-\Theta(d)}$. 
\end{corollary}

\begin{remark}
For  $d\ge 2$  the result of Corollary~\ref{Cor_Concentration} holds also for Latin hypercube samples instead of Monte Carlo samples. This follows from Proposition~\ref{Prop:Expec_Disc} and the results proved in \cite{GH16} and \cite{DDG18}.
\end{remark}

\begin{remark}\label{Rem:Matousek}
Theorem~\ref{thm:improved_star_discrepancy_bound}, Proposition~\ref{Prop:Expec_Disc},  and Corollary~\ref{Cor_Concentration} show that the discrepancy of an $N$-point sampling scheme (consisting of i.i.d. or negatively dependent points) in $[0,1]^d$ has typically (on average and with very high probability) a discrepancy of order at most $\sqrt{d/N}$. Now in practice it is desirable to have sampling schemes that are extensible in the number of points $N$ or even in the dimension $d$. The results presented in this paper do not guarantee that the proved discrepancy bounds still hold with high probability if we keep adding random points to our sampling scheme or additional random components to our points to obtain higher-dimensional point sets. 
Results that take care about these issues were first presented by Dick \cite{Dic07}, who proved probabilistic discrepancy bounds for infinite-dimensional infinite sequences of 
independent and uniformly distributed random points; these results were slightly improved in \cite{DGKP08}. Later, Aistleitner \cite{Ais13} and Aistleitner and Weimar \cite{AW13} substantially improved these bounds and extended them. To provide a specific example, it was proved in \cite{AW13} that with positive probability a random matrix $X\in [0,1]^{\N\times \N}$ satisfies for all $d\in \N$ and all $N \ge 2$ the discrepancy bound 
\begin{equation}\label{est:aw13}
D^*_N(P_{N,d}) \le \sqrt{ 2130 + 308 \frac{\log \log (N)}{d}} \, \sqrt{\frac{d}{N}}\,,
\end{equation} 
where $P_{N,d}\subset ([0,1]^d)^N$ denotes the projection of $X$ onto its first $d$ rows and first $N$ columns. Note that compared to \eqref{est:disc*}, 
 an additional term $\sqrt{\log \log (N)/d}$ appears in \eqref{est:aw13}. This cannot be avoided: although the typical discrepancy of random points is $\sqrt{d/N}$, the law of the iterated logarithm states that for fixed $d$ we have
$$
\limsup_{N\to \infty}\frac{\sqrt{N}\,D^*_N (P_{n,d})}{\log\log(N)} \ge \frac{1}{\sqrt{2}}\,.
$$
Note also that the constants in \eqref{est:aw13} are rather large.
We are convinced that we are able to improve the constants in the bounds from \cite{AW13} (and probably also from \cite{Ais13}) reasonably by using our improved bounds for $\delta$-covers from Theorem~\ref{thm:general_bracketing_number} 
and some of the tweaks used in the proofs in \cite{GH16}.  But since this would require additional effort and this paper is already rather long, we leave this as a potential topic of future research (cf. also Section~\ref{SEC:Future_Research}).
\end{remark}

Since proving that a given sampling scheme is $\mathcal{D}_0^d$-$e^{\rho d}$-negatively dependent is in some cases a difficult task, we now consider the relaxed condition of $\mathcal{C}_0^d$-$e^{\rho d}$-negative dependence. With the help of Theorem~\ref{thm:general_bracketing_number}, we can also improve the existing bounds in that case (\cite{WGH19}, Theorem~3.1).

\begin{corollary} Let $d,N \in \mathbb{N}$ and $\rho \in [0,\infty)$. Let $X = (X_j)_{j=1}^N$ be a $\mathcal{C}_0^d$-$e^{\rho d}$-negatively dependent sampling scheme in $[0,1)^d$. Then for every $c > 0$
\begin{align} \label{eq5}
	D_N^*(X) \leq c \sqrt{\frac{d}{N} \max \left( 1 , \log \left( \frac{N}{d} \right) \right)	 }
\end{align}
holds with probability at least $1 - 2 \exp\left( (-\tfrac{1}{2}(c^2-1) \xi + \rho + \log(e(\tfrac{2}{c}+1)) + \tfrac{\log(b_d)}{d})d \right)$, where $\xi := \max \left( 1 , \log \left( \tfrac{N}{d} \right) \right)$ 
and $b_d := \max(1, 1.1^{d-101})$. 
Moreover, for every $\theta \in (0,1)$
\begin{align} \label{eq6}
	\mathds{P}\left( D_N^*(X) \leq \sqrt{\frac{2}{N}} \sqrt{d \log(\eta) + \rho d + \log \left( \frac{2}{1-\theta} \right)} \right) \geq \theta,
\end{align} 
where $$\eta:=\eta(N,d) = 3.3e \left( \max \left(1,\frac{N}{2d\log(3.3e)}\right) \right)^{\frac{1}{2}}.$$
\end{corollary}

\begin{proof} 
The proof relies on an intermediate result of Theorem~3.1 in \cite{WGH19}, namely the inequality
\begin{equation}\label{intermediate_result}
\mathds{P}(D_N^*(X) < 2\delta) \geq 1 - 2e^{\rho d} |\Gamma| e^{-2N\delta^2}
\end{equation}
that holds for all $\delta >0$; here $\Gamma$ is the cardinality of an arbitrary $\delta$-cover.
Due to Theorem \ref{thm:general_bracketing_number} and Sterling's formula we have
$$|\Gamma| \le N(d, \delta) \le 2N_{[\ ]}(d, \delta) < \max\left({1.1}^{d-101},1\right) e^d (\delta^{-1} + 1)^d,$$
since 
$$2\frac{d^d}{d!} \leq 2\frac{e^d}{\sqrt{2\pi d}} < e^d $$
for all $d\in \mathbb{N}$. Now, we deduce \eqref{eq6} from \eqref{intermediate_result} and follow exactly the lines of the proof of Theorem~3.1 in \cite{WGH19}; the only exception is that we 
end up with the condition
$$\left( \frac{\eta}{b_d e} - 1 \right) \log(\eta)^{\frac{1}{2}} > \left( \frac{2N}{d} \right)^{\frac{1}{2}},$$
which is fulfilled by the $\eta$ in the statement of the result. Then, we apply \eqref{intermediate_result} with $2\delta = c \sqrt{\frac{d}{N}\xi}$ to prove \eqref{eq5}.
\end{proof}

\subsection{Bounds for the Weighted Star-Discrepancy}\label{SUBSEC:Weighted_Star_Discrepancy}

Finally, we also treat the case of the weighted star-discrepancy and improve the quantitative bounds from \cite{WGH19}, Theorem~3.4, and \cite{Ais14}, Theorem~1. For that purpose, we shortly recall the definition of weighted star-discrepancy as introduced in \cite{SW98}: The weights are a collection of non-negative numbers $\gamma = (\gamma_u)_{\emptyset \neq u \subseteq [d]}$, where $\gamma_u$ is interpreted as the weight of the coordinates from the subset $u$. Moreover, recall from our presentation of the Koksma-Hlawka inequality that $(x_u,1)$ denotes the point in $[0,1]^d$ whose $i$-th component is $x_i$ if $i \in u$ and $1$ otherwise. The \textbf{weighted star-discrepancy} of a $N$-point set $P$ and weights $\gamma$ is defined by
$$D^*_{N,\gamma}(P) := \sup_{z \in [0,1]^d} \max_{\emptyset \neq u \subseteq [d]} \gamma_u |D_N(P,(z_u,1))|.$$
More information about the weighted star discrepancy can, e.g., be found in \cite{Ais14, HPS08}. The following proposition leads to an improvement of \cite{Ais14}, Theorem~1, see Corollary~\ref{Ais_weighted_Discrepancy}.

\begin{proposition}\label{Prop:Pre_Aistleitner}
Let $d,N \in \mathbb{N}$ and $\alpha, \beta \in [0,\infty)$.
Let $X = (X_j)_{j=1}^N \subset [0,1)^d$ be a family of random points
such that for every $\emptyset \neq u \subseteq [d]$ its projection $X^u$ onto the coordinates in $u$ satisfies
 for every $q \in (0,1)$
	\begin{align}\label{cond:weighted_star} 
	D_N^*(X^u) \leq \alpha \sqrt{\beta + \frac{\log((1-q)^{-1})}{|u|}} \sqrt{\frac{|u|}{N}}
	\end{align}
	with probability at least $q$. Then there exists a realization $P\subset [0,1]^d$ of $X$ such that
\begin{align}\label{improved_aistleitner} 
	D_{N,\gamma}^*(X) \leq \frac{\alpha}{\sqrt{N}} 
	 \max_{\emptyset \neq u \subseteq [d]} \gamma_u  \sqrt{\beta + 1+ \log \left( 1 + \frac{1}{\sqrt{2\pi}} \right) + \log(d) - \left( 1+ \frac{1}{2|u|} \right) \log(|u|) }\,
	\sqrt{|u|}.
\end{align}
\end{proposition}

\begin{proof}
For $\emptyset \neq u \subseteq d$ we put
$$q := 1 - e^{-|u| \left( \log \left( \frac{ed}{|u|} \right) + \log \left( 1+\frac{1}{\sqrt{2\pi}} \right) \right) + \frac{1}{2} \log(|u|)}$$
and
$$A_u := \left\{ D_N^*(X^u) > \alpha \sqrt{ \beta + \frac{\log \left( (1-q)^{-1} \right)}{|u|}}
\sqrt{\frac{|u|}{N}} \right\}.$$
To prove Proposition~\ref{Prop:Pre_Aistleitner} it suffices to show that 
$P \left(\cup_{\emptyset \neq u \subseteq [d]} A_u \right) < 1$.
Indeed, using
$$ {d\choose r} \le \frac{1}{\sqrt{2\pi r}} \left( \frac{ed}{r} \right)^r 
\hspace{3ex}\text{for every $1\le r\le d$,}$$
we obtain due to \eqref{cond:weighted_star} 
\begin{equation*}
\begin{split}
\mathds{P} \left(\bigcup_{\emptyset \neq u \subseteq [d]} A_u \right) &\le \sum_{r=1}^d \sum_{|u| = r} \mathds{P} (A_u) \le \sum_{r=1}^d  { d \choose r} e^{-r \left( \log \left( \frac{ed}{r} \right) + \log \left( 1+\frac{1}{\sqrt{2\pi}} \right) \right) + \frac{1}{2} \log(r)}\\
&\le \frac{1}{\sqrt{2\pi}} \sum_{r=1}^d e^{-r \log \left( 1+\frac{1}{\sqrt{2\pi}} \right) }
<  \frac{1}{\sqrt{2\pi}} \left( \frac{1}{1- \left( 1+\frac{1}{\sqrt{2\pi}} \right)^{-1}} -1 \right) = 1.
\end{split}
\end{equation*}
\end{proof}

We apply Proposition~\ref{Prop:Pre_Aistleitner} to Monte Carlo point sets. Then
Theorem~\ref{thm:improved_star_discrepancy_bound} with $\rho=0$ gives us immediately the following corollary.

\begin{corollary}\label{Ais_weighted_Discrepancy}
There exists a point set $P\subset [0,1)^d$ such that 
\begin{align}\label{improved_aistleitner2} 
	D_{N,\gamma}^*(X) \leq \frac{0.7723}{\sqrt{N}} 
	 \max_{\emptyset \neq u \subseteq [d]} \gamma_u  \sqrt{11.78864+ \log(d) - \left( 1+ \frac{1}{2|u|} \right) \log(|u|) }\,
	\sqrt{|u|}.
\end{align}
\end{corollary}

Finally, we improve  \cite{WGH19}, Theorem~3.4.

\begin{corollary}\label{Cor:WGH19+} 
Let $d,N \in \mathbb{N}$ and let $X = (X_j)_{j=1}^N \subset [0,1)^d$ be a sampling scheme such that for every $\emptyset \neq u \subseteq [d]$ its projection onto the coordinates in $u$ is $\mathcal{D}_0^{|u|}$-$e^{\rho|u|}$-negatively dependent. Then for any weights $(\gamma_u)_{\emptyset \neq u \subseteq [d]}$ and any $c>0$ we obtain
	\begin{align*} D_{N,\gamma}^*(X) \leq c \max_{\emptyset \neq u \subseteq [d]}  \gamma_u \sqrt{\frac{|u|}{N}}
	\end{align*}
with probability at least $2 - (1 + e^{-(1.67681c^2-10.45292-\rho)})^d$. For $\theta \in (0,1)$ we have
\begin{align} \label{ineq:3}
	\mathds{P}\left(  D_{N,\gamma}^*(X) \leq
	\sqrt{\frac{\rho + 10.45292 - \log((2-\theta)^{1/d}-1)}{1.67681}}
	  \max_{\emptyset \neq u \subseteq [d]} \gamma_u  \sqrt{\frac{|u|}{N}} \right) \geq \theta\,.
\end{align}
\end{corollary}

The proof can be done exactly along the lines of the one of \cite[Theorem~3.4]{WGH19}.

\section{Future Research}\label{SEC:Future_Research}

In this last short section we would like to mention potential future research. It would be interesting to further improve our upper bounds on the bracketing number and to provide explicit bracketing and $\delta$-covers that satisfy those bounds and that can be efficiently computed.\\[12pt]
Clearly, with our bracketing result Theorem~\ref{thm:general_bracketing_number} 
or even additional refinements one may, similar as in Theorem~\ref{thm:improved_star_discrepancy_bound}, further improve other probabilistic discrepancy or dispersion bounds and extend them from Monte Carlo point sets to the larger class of $\gamma$-negatively-dependent random point sets. For instance, by using Theorem~\ref{thm:general_bracketing_number} and the dyadic chaining method from \cite{Ais11} one may substantially improve the results for the extreme discrepancy from \cite[Theorem~2.2]{Gne08}. Similarly, using Theorem~\ref{thm:general_bracketing_number} and some additional tweaks from the proofs in \cite{GH16} should lead to much smaller constants in the discrepancy bounds for the (infinite-dimensional infinite) sequences of points presented in \cite{Ais14, AW13}. 
Such an improvement would be valuable in practice, since for numerical applications it is favorable to have sequences of sample points instead of sheer sample sets that are not extensible in the number of points, cf. Remark~\ref{Rem:Matousek} and the remarks on multilevel algorithms below.\\[12pt]

Moreover, new discrepancy bounds and better $\delta$-covers may be used to improve algorithms for calculating discrepancies and for constructing point sets and sequences with small discrepancies, cf. \cite{DGW09, DGW10, DGW14, GGP20, Thi01}. It would also be interesting to investigate whether good estimates for the dispersion in combination with
better $\delta$-covers lead to competitive algorithms for calculating the dispersion of given point sets or to construct point sets with low dispersion (cf. also \cite{UV19}). 
Note that algorithms that compute the discrepancy (or dispersion) of given point sets may, in particular, be used for a ``semi-construction'' of low-discrepancy (or low-dispersion) sets: One may randomly generate a point set that has low discrepancy (or dispersion) with high probability and compute afterwards its discrepancy (or dispersion). If it is small enough, one accepts the point set, otherwise one rejects the point set and generates a new one, and so on.\\[12pt]
A more ambitious goal is to verify that certain structure-preserving randomizations of low-discrepancy point sets are $\gamma$-negatively dependent and therefore satisfy the probabilistic pre-asymptotic discrepancy bounds from Theorem~\ref{thm:improved_star_discrepancy_bound} 
with high probability. Due to the structure-preserving randomization, these points then have essentially optimal asymptotic discrepancy bounds for sure. Such point sets would be extremely  useful for numerical integration in high or infinite-dimensions. 
A particularly thrilling perspective would be to use them as integration nodes of randomized quasi-Monte Carlo algorithms that serve as building blocks for sophisticated adaptive multilevel algorithms (cf., e.g., \cite{BG14, GW09, HMNR10}): the resulting building block algorithms would from a probabilistic point of view never be worse than pure Monte Carlo algorithms (even for levels of high dimension where only a small number of samples is used), but on levels of lower dimension, where a large number of samples is used, their superior asymptotic behavior would automatically kick in.
Note that we are referring to adaptive algorithms, i.e. the meaning of \textit{low} and \textit{high} dimension as well as of \textit{small} and \textit{large} numbers of samples are relative. They may change while the algorithm is running and investing more and more resources and introducing more and more levels. That means, we cannot say beforehand whether the building block algorithms for a specific level should rather exhibit a good asymptotic behavior or preferably a good pre-asymptotic behavior. That is why building block algorithms that satisfy good asymptotic as well as pre-asymptotic bounds are highly desirable.\\[12pt]
\paragraph*{Acknowledgment} The authors would like to thank Friedrich Pillichshammer for his question which motivated Corollary~\ref{cor:expectation}, and the anonymous referees for their valuable comments. Furthermore, the first-named and the third-named author would like to thank RICAM in Linz for their hospitality where they got to know each other during the special semester on Multivariate Algorithms and their Foundations in Number Theory.

\end{document}